\documentclass{amsart}

\usepackage{amsthm}
\usepackage{enumerate}
\usepackage{amsmath,amssymb,amscd,amsfonts}
\usepackage{mathtools}
\usepackage{bbm}
\usepackage{color}
\usepackage{lineno}

\allowdisplaybreaks[4]

\newtheorem{theorem}{Theorem}[section]
\newtheorem{lemma}[theorem]{Lemma}

\newtheorem{corollary}[theorem]{Corollary}

\theoremstyle{definition}

\theoremstyle{remark}
\newtheorem{remark}[theorem]{Remark}

\numberwithin{equation}{section}

\newcommand{\N}{\mathbb{N}}

\newcommand{\R}{\mathbb{R}}
\newcommand{\C}{\mathbb{C}}

\newcommand{\mP}{\mathbb{P}}

\newcommand{\1}{\mathbbm{1}}

\newcommand{\cB}{\mathcal{B}}

\newcommand{\cF}{\mathcal{F}}

\newcommand{\cL}{\mathcal{L}}

\newcommand{\al}{\alpha}
\newcommand{\be}{\beta}
\newcommand{\ga}{\gamma}
\newcommand{\del}{\delta}
\newcommand{\la}{\lambda}
\newcommand{\ep}{\varepsilon}
\newcommand{\vp}{\varphi}
\newcommand{\sh}{\theta}

\newcommand{\om}{\omega}
\newcommand{\Om}{\Omega}
\newcommand{\hOm}{\hat{\Omega}}
\newcommand{\hmP}{\hat{\mathbb{P}}}
\newcommand{\hbe}{\hat{\beta}}
\begin{document}

\title[random dynamics of beta-transformations]{Absolutely continuous invariant measures for random dynamical systems of beta-transformations}



\author{Shintaro Suzuki}
\address{Department of Mathematics, Tokyo Gakugei University, Nukuikita 4-1-1, Koganei, Tokyo 184-8501, Japan}
\email{shin05@u-gakugei.ac.jp}


\begin{abstract}
We consider an independent and identically distributed (i.i.d.) random dynamical system of simple linear transformations on the unit interval $T_{\beta}(x)=\beta x$ (mod $1$), $x\in[0,1]$, $\beta>0$, which are the so-called beta-transformations. For such a random dynamical system, including the case that it is generated by uncountably many maps, we give an explicit formula for the density function of a unique stationary measure under the assumption that the random dynamics is expanding in mean. As an application, in the case that the random dynamics is generated by finitely many maps and the maps are chosen according to a Bernoulli measure, we show that the density function is analytic as a function of parameter in the Bernoulli measure
and give its derivative explicitly. 
Furthermore, for a non-i.i.d. random dynamical system of beta-transformations, we also give an explicit formula for the random densities of a unique absolutely continuous invariant measure under a certain strong expanding condition or under the assumption that the maps randomly chosen are close to the beta-transformation for a non-simple number in the sense of parameter $\be$.

\end{abstract}
\maketitle
\begin{section}{Introduction}
Deterministic dynamical systems generated by a piecewise smooth expanding map on the unit interval have been intensively investigated as a simple model of  the so-called chaotic dynamical systems. A key object in order to analyze the ergodic properties of such a dynamical system is an absolutely continuous (with respect to the Lebesgue measure) invariant probability measure.
One of the methods to ensure its existence is to
use the associated Perron-Frobenius operator, 
based on the fact that a non-negative integrable function whose integral equals to $1$ is a fixed point of the associated operator if and only if it is an invariant density, i.e., the density function of some absolutely continuous invariant probability measure (see \cite{La-Yo}). 
The method via Perron-Frobenius operators is also efficient in the context of random dynamical systems generated by piecewise smooth maps on the interval, which was established in \cite{Bu}, \cite{Mo1}, \cite{Pe} and \cite{In} for example.
Compared with a deterministic system, a random dynamical system is represented by a pseudo skew-product map and a natural reference measure in this setting is the product of the measure on the `noise space' and the Lebesgue measure on the unit interval. 

A key ingredient in constructing an invariant density via the associated Perron-Frobenius operator is to show sequential compactness of a sequence of functions given by $n$-iterations of the associated operator applied to some positive function. This yields that an invariant density is obtained as the limit of some subsequence of it. In general, we have no explicit formula for the limit and no knowledge to a closed form of the invariant density, which provide a motivation to investigate its explicit formula in anothor way.

For deterministic piecewise linear expanding maps, an explicit formula for an invariant density can be given. For beta-transformations, R\'enyi  \cite{Re} gave an explicit formula for a unique invariant density in case that parameter $\be$ is  the golden ratio. Parry \cite{Pa1} and Gel'fond \cite{Ge} independently extended the result to any cases. The corresponding result to various piecewise linear expanding maps has been given, for example, for linear unimodal maps \cite{It-Ta}, linear mod $1$ maps \cite{Pa2}, generalized beta-maps \cite{Go1} and so on. For more general cases, Kopf \cite{Ko} obtained an explicit formula for any invariant density of a piecewise linear eventually expanding map, which is given by the determinant of a matrix each of whose component is some function having information about the kneading sequences for critical points of the map. See \cite{Go2} for an analogue result.
  
In light of deterministic cases, it is natural to ask whether we can extend the results for an explicit formula for an invariant density to random cases. The first related result as far as the author knows was given by Tanaka and Ito \cite{Ta-It}, who treated a random dynamics generated by two linear unimodal expanding maps, which are chosen according to the $(1/2,1/2)$-Bernoulli measure. Related to number expressions of base $\be>1$, Dajani and Kraaikamp \cite{Da-Kr} introduced the random beta-transformation generated by two piecewise linear expanding maps each of whose slope is $\be>1$, which are chosen according to a Bernoulli measure. Dajani and de Vries \cite{Da-de} gave an explicit formula for the density function of a unique stationary measure in the case that the greedy expansion of $1$ is finite. Kempton \cite{Ke} gave the corresponding result in the case of any $\beta>1$ and  the $(1/2,1/2)$-Bernoulli measure, by constructing the natural extension of the random map. In \cite{Su}, the author gave an explicit formula for the density function in any cases by solving a functional equation arising from the associated Perron-Frobenius operator. For more general cases, Kalle and Maggioni \cite{Ka-Ma} extended Kopf's method to random dynamical systems 
and gave an explicit formula for the density function of any stationary measure under some mild conditions in the case of those generated by countably many piecewise linear maps which are chosen according to a Bernoulli measure.


In this paper, we consider random dynamical systems of beta-transformations and 
establish some explicit formulas related to their unique invariant densities. 
First, for an i.i.d. random dynamics of beta-transformations, we give an explicit formula for the density function of a unique stationary measure under the assumption that a random dynamical system is expanding in mean (Theorem \ref{Thm A}). 
The novelty of this result is that it includes the case of random dynamical systems generated by uncountably many maps. In particular, we can apply the result to an i.i.d. (not necessary small)  perturbation of a beta-transformation on parameter $\be$.  As an application of Theorem \ref{Thm A}, we give the supremum and infimum of
the density function (Corollary \ref{cor1}). 

Second, as another application of Theorem \ref{Thm A}, we establish a linear response formula for a unique stationary measure in the case that the random dynamics is generated by finitely many beta-transformations and the maps are chosen according to a Bernoulli measure. In detail, we show the analyticity of the stationary measure as a function of the parameter in the Bernoulli measure and give its derivative explicitly (Theorem \ref{Thm B}). 

Finally, for a non-i.i.d. random dynamical system of beta-transformations, we also establish an explicit formula for a unique invariant density, which is the first result in the context of non-i.i.d. random dynamical systems generated by uncountably many maps. We give the explicit formula in the case that the maps chosen are strongly expanding in the sense that each of their slopes is greater than $2$ (Theorem \ref{Thm C}) or in the case that the maps chosen are close to the beta-transformation for a non-simple number in the sense of parameter $\be$ (Theorem \ref{Thm D}), which are done by solving certain functional equations (the equations (\ref{PF}) in Section $5$) arising from samplewise Perron-Frobenius operators.

This paper is organized as follows. In Section 2, we introduce random dynamical systems of beta-transformations and summarize their basic properties. 
In Section 3, we give an explicit formula for the density function of the unique stationary measure of an i.i.d. random dynamics of beta-transformations. 
In Section 4, we show a linear response formula for a unique stationary measure in case that a random dynamics is generated by finitely many beta-transformations and the maps are chosen according to a Bernoulli measure. In Section 5, for a non-i.i.d. random dynamics of beta-transformations, we establish an explicit formula for random densities in the case that the slope of each map randomly chosen is greater than $2$  or in the case that the maps chosen are close to the beta-transformation for a non-simple number.
 
\end{section}

\begin{section}{Preliminaries}
In this section, we summarize some notions used throughout the paper. In subsection 2.1. we introduce random dynamical systems of beta-transformations. In subsection 2.2. we see that  multiple base expansions of $x\in[0,1]$, whose properties have been investigated recently (e.g., \cite{Li}  and \cite{Ne}), are generated by such a random dynamical system. In particular, we give a sufficient condition to guarantee `almost all' expansions actually converge to $x\in[0,1]$. In subsection 2.3, we introduce the Perron-Frobenius operator for a beta-transformation and show a key lemma in order to give the main results in the paper. 

\begin{subsection}{Random dynamics of beta-transformations}
Let $\al>0$ be a positive number. We define the map $T_\al:[0,1]\to[0,1]$ by
\[T_\al(x)=\al x -[\al x]\]
for $x\in[0,1]$, where $[y]$ denotes the integer part of $y\in\R$. In case of $\al>1$
the map is expanding and has a unique invariant probability measure absolutely continuous with respect to 
the Lebesgue measure, whose density function can be given explicitly (see \cite{Pa1}, \cite{Re} and \cite{Ge}).

Let $(\Om, \cF, \mP)$ be a Lebesgue space and $\sh:\Om\to\Om$ be an ergodic measure preserving map on it. We call $(\Om, \cF, \mP)$ the noise space and  $\sh:\Om\to\Om$ the noise dynamics. Let $\be:\Omega\to(0,+\infty)$ be a positive random variable. We define the map $\tau:\Omega\times [0,1]\to [0,1]$ by $(\om,x)\mapsto T_{\be(\om)}(x)$. Then the random dynamical system can be represented by the skew-product map $R(\om,x)=(\sh\om,\tau(\om,x))$ for $(\om,x)\in\Om\times[0,1]$.
For simplicity, we write $\tau_\om(\cdot)=\tau(\om,\cdot)$ for $\om\in\Om$.
By setting $\tau_\om^0=id$ and $\tau_\om^n=\tau_{\sh^{n-1}\om}\circ\dots\tau_\om$ for $\om\in\Om$ and $n\geq1$, where $id$ denotes the identity map on $[0,1]$, we can write the $n$-th iteration of the skew product map as
\[R^n(\om,x)=(\sh^n\om, \tau_\om^n(x))\]
for $(\om,x)\in\Om\times[0,1]$. 
\end{subsection}

\begin{subsection}{Random multiple base expansions}
As a generalization of beta-expansions generated by a single beta-transformation, the random dynamical systems as defined above generate multiple base expansions of $x\in[0,1]$ as follows. Set $\be_\om^{(0)}=1$ and 
\[\be_\om^{(n)}=\prod_{i=0}^{n-1}\be(\sh^{i}\om)\]
for $\om\in\Om$ and $n\geq1$. We write $\be_\om=\be_\om^{(1)}=\be(\om)$ for simplicity. 
Define the digit function by $d_n(\om,x)=[\be_{\sh^{n-1}\om}\tau_{\om}^{n-1}(x)]$ for $(\om,x)\in\Om\times[0,1]$ and $n\geq1$. Then by the definition of the map $\tau$, we have 
\[x=\frac{[\be_\om x]}{\be_\om}+\frac{\tau_\om(x)}{\be_\om}\]
for $(\om,x)\in\Om\times[0,1]$. By the fact that 
\[\tau_{\om}^{n}(x)=\frac{[\be_{\sh^{n}\om} \tau_{\om}^{n}(x)]}{\be_{\sh^n\om}}+\frac{\tau_\om^{n+1}(x)}{\be_{\sh^n\om}}\]
for $n\geq0$, we obtain
\begin{linenomath}
\begin{align*}
x&=\frac{[\be_\om x]}{\be_\om}+\frac{\tau_\om(x)}{\be_\om} \\
 &=\frac{[\be_\om x]}{\be_\om}+\frac{[\be_{\sh\om} \tau_\om(x)]}{\be_{\om}\be_{\sh\om}}+\frac{\tau^2_\om(x)}{\be_\om \be_{\sh\om}} \\
&=\cdots \\
&=\sum_{n=1}^N\frac{d_n(\om,x)}{\be_\om^{(n)}}+\frac{\tau_\om^{N}(x)}{\be_{\om}^{(N)}}
\end{align*}
\end{linenomath}
for $N\geq2$. Note that $\tau_\om^{n}(x)\in[0,1]$ for $n\geq0$. Then under the assumption that 
\[(\be_\om^{(n)})^{-1}\to0 \]
as $n\to\infty$, we have 
\[x=\sum_{n=1}^\infty\frac{d_n(\om,x)}{\be_\om^{(n)}}.\]
In the following, we give a sufficient condition which guarantees 
$(\be_\om^{(n)})^{-1}\to0$ as $n\to\infty$ for $\mP$-a.e. $\om\in\Om$.

\begin{lemma}\label{expansion}
Assume that 
\begin{equation}\label{mean expanding}
\int_{\Om}\log \frac{1}{\be(\om)}d\mP<0.
\end{equation}
Then for $\mP$-a.e. $\om\in\Om$ we have 
\[(\be_\om^{(n)})^{-1}\to0\]
as $n\to\infty$.
\end{lemma}

\begin{remark}
In the case that the noise dynamics $\sh:\Om\to\Om$ is invertible, the condition (\ref{mean expanding}) is also a sufficient condition to guarantee  the existence of a unique  invariant probability measure for the skew-product map $R$ absolutely continuous with respect to $\mP\times l$ (see \cite{Bu}).
\end{remark}

\begin{proof}
By applying the Birkhoff ergodic theorem to the function $\log(1/ \be(\om))$, we see that
\[\frac{1}{n}\sum_{i=0}^{n-1}\log(\be(\sh^i\om))^{-1}\to \int_{\Om}\log \frac{1}{\be(\om)}d\mP\] 
for $\mP$-a.e. $\om\in\Om$, which yields that there is some positive constant $C_\om>0$ such that 
\[\be_\om^{(n)}=\prod_{i=0}^{n-1}\be(\sh^i\om)\geq C_\om \exp\Biggl(-\frac{1}{2}\int_{\Om}\log \frac{1}{\be(\om)}d\mP\cdot n\Biggr) \to \infty\]
as $n\to\infty$ for $\mP$-a.e. $\om\in\Om$. This finishes the proof. 
\end{proof}

\end{subsection}

\begin{subsection}{Samplewise Perron-Frobenius operators}
Let us denote by $l$ the Lebesgue measure on $[0,1]$ and by $(L^1(l), ||\cdot||_1)$ the Banach space of integrable functions with respect to the Lebesgue measure 
$l$. For a positive constant $\alpha>0$, we define the Perron-Frobenius operator for  the corresponding map $T_\al$ by
\[\cL_\al f(x)=\frac{1}{\al}\sum_{y ; x=T_\al(y)} f(y)\]
for $x\in[0,1]$ and $f\in L^1(l)$. This operator is linear and bounded on $L^1(l)$ with $||\cL_\alpha f||_1=||f||_1$ for $f\in L^1(l)$, whose spectral radius is $1$ (see \cite{Ba}, \cite{Bo-Go} and \cite{La-Yo} for example). 

For a function $f:[0,1]\to\R$ denote by $\bigvee f$ the total variation of $f$. We define 
\[|f|_{BV}=\inf\Biggl\{\bigvee f^* ; f^* \text{is a version of} \ f\Biggr\}.\]
Set $BV=\{f\in L^1(l)\ ;\ |f|_{BV}<+\infty\}$ and
\[||f||_{BV}=||f||_1+|f|_{BV}\]
for $f\in BV$. Then $(BV, ||\cdot||_{BV})$ is a Banach space and the Perron-Frobenius operator is  linear and bounded on $(BV, ||\cdot||_{BV})$ 
(e.g., \cite{Ba}, \cite{Bo-Go}, \cite{La-Yo} and \cite{Li-Yo}). In the case of $\alpha>1$, we know that the  map $T_\alpha$ is uniformly expanding and the corresponding Perron-Frobenius operator $\cL_\alpha$ is quasi-compact, i.e., for any spectral value $\la\in\C$ whose modulus is greater than $1/\al$ is an isolated eigenvalue with finite multiplicity, which plays an important role in investigating the ergodic properties of $T_\al$ (e.g., \cite{Ba}, \cite{Bo-Go}, \cite{La-Yo} and \cite{Li-Yo}). 

In the following, we write $\cL_\om=\cL_{\be_\om}$ for $\om\in\Om$. The next lemma is a key fact to show the main results in the paper.

\begin{lemma}\label{key}
Let $a\in[0,1]$ and $\om\in\Om$. Then 
\[\cL_\om \1_{[0,a]}=\frac{[\be_\om a]}{\be_\om}\1_{[0,1]}+\frac{1}{\be_\om}\1_{[0,\tau_\om(a)]},\]
where $\1_A$ denotes the indicator function of $A$.

\end{lemma}
\begin{proof}
Since the image of the map $\tau_\om$ on $[i/\be_\om, (i+1)/\be_\om)$ is $[0,1)$ for $0\leq i\leq [\be_\om]-1$, by definition we have  
\begin{linenomath}
\begin{align*}
\be_\om\cL_\om\1_{[0,a]}
&=\#\{y\in[0,a]\ ;\ \tau_\om(y)=x\} \\
&=\begin{cases}
[\be_\om a]+1 & \text{ if } x\in(0,\tau_\om(a)), \\
[\be_\om a] & \text{ if } x\in (\tau_\om(a), 1).
\end{cases}
\end{align*}
\end{linenomath}
Together with the fact that for any $t\in(0,1)$ the functions $\1_{(0,t)}$ and $\1_{[0,t]}$ has the same version in $L^1(l)$, we obtain the conclusion.

\end{proof}

\end{subsection}
\end{section}

\begin{section}{Invariant densities for i.i.d. cases }
This section is devoted to giving an explicit formula for the density function of a unique stationary 
measure for i.i.d. random dynamical systems of beta-transformations. Throughout this section we assume that the noise space is given by the product space $(\Om, \cF, \mP)=(\hat{\Om}, \hat{\cF}, \hat{\mP})^{\N}$ of some Lebesgue space $(\hat{\Om}, \hat{\cF}, \hat{\mP})$ and the noise dynamics is the left shift $\theta:\Om\to\Om$, i.e., 
$\theta(\{\om_i\}_{i=1}^\infty)=\{\om_{i+1}\}_{i=1}^\infty$ for $\{\om_{i}\}_{i=1}^\infty\in\Om$. Let $\hat{\be}:\hat{\Om}\to(0,+\infty)$ be a measurable function and define the function $\be:\Om\to(0,+\infty)$ by
$\be(\{\om_{i}\}_{i=1}^\infty)=\hat{\be}(\om_1)$ for $\{\om_{i}\}_{i=1}^\infty\in\Om$. Then the function $\be$ is a positive measurable function on $(\Om, \cF)$ and the random variables $\{\be\circ\sh^n\}_{n=0}^\infty$ are i.i.d. on $(\Om, \cF, \mP)$. 

Let us define the sample averaged (or annealed) Perron-Frobenius operator $\cL: L^1(l)\to L^1(l)$ by 
\[\cL f=\int_{\Om}(\cL_\om f) d\mP\]
for $f\in L^1(l)$. We note that this operator is well-defined by the inequality 
\[||\cL_\om f||_1\leq\int_0^1\int_\Om|\cL_\om f| d\mP dl\leq\int_\Om\int_0^1|\cL_\om f| dl d\mP\leq||f||_1\]
for $f\in L^1(l)$, which also ensures that the operator $\cL$ is bounded. The fact that the operator $\cL$ is linear follows from its definition. A key ingredient in order to show the existence of an $R$-invariant probability measure absolutely continuous with respect to the product measure $\mP\times l$ is that for $h\in L^1(l)$ with $h\geq0$ and $||h||_1=1$ the probability measure $\mP\times hl$ is $R$-invariant if and only if $h$ satisfies $\cL h=h$. 
Furthermore, since the random variables $\{\be\circ\sh^n\}_{n=0}^\infty$ are i.i.d. on $(\Om, \cF, \mP)$, we know that an $R$-invariant probability measure $\mu$ absolutely continuous with respect to the product measure $\mP\times l$ has the product form $\mu=\mP\times m$ (Theorem 3.1 in \cite{Mo2}), where $m$ is a probability measure absolutely continuous with respect to $l$.
This yields that the density function of any absolutely continuous invariant probability measure of $R$ is given by the form $\1_\Om\times h$, where $h$ is a non-negative normalized eigenfunction of $\cL$ associated to an eigenvalue $1$.

We assume that 
\[\int_\Om\frac{d\mP(\om)}{\be(\om)}=\int_{\hat{\Om}}\frac{d\hat{\mP}(a)}{\hat{\be}(a)}<1.\]
By applying Lemma 6.7 in \cite{In} to our setting, we have that  $\cL$ is well-defined as a linear bounded operator on $BV$ and it is quasi-compact, i.e., any spectral value $\la\in\C$ whose modulus is greater than $\int_\Om d\mP(\om)/\be(\om) $ is an isolated eigenvalue with finite multiplicity. In addition, $1$ is actually this eigenvalue of $\cL$, which ensures the existence of an $R$-invariant probability measure $\mu$ absolutely continuous with respect to $\mP\times l$. 

In fact, we can see the uniqueness of such a measure as follows. By applying to Theorem 5.6 in \cite{In}, we have that there exist a positive integer $N$, a finite set of $BV$ functions $\{h_i\}_{i=1}^N$, a finite set of linear functionals $\{\eta_i\}_{i=1}^N$ and a linear bounded operator $Q:L^1(l)\to L^1(l)$ such that
\[\cL f=\sum_{n=1}^N \eta_{n}(f) h_n+Qf\]
for $f\in L^1(l)$, where $h_i h_j=0$ for $i\neq j$, $\cL h_i=h_{i+1}$ for $1\leq i< N$, $\cL h_N=h_1$
and $||Q||<1$. We note that the support of a function of bounded variation is represented by a countable union of open intervals. Let $h$ be an element of $\{h_i\}_{i=1}^N$ and take an open interval $I$ included in the support of $h$. 
By Lemma \ref{expansion}, for almost all $\om\in\Om$ we have that $\be_\om^{(n)}\to\infty$ as $n\to\infty$, which ensures that for almost all $\om\in\Om$ there is a positive integer $n(\om)$ such that $\be(\theta^{n(\om)}\om)>1$ and $\tau_\om^{n(\om)}(I)\cap\{1/\be(\theta^{n(\om)}\om),\dots,[\be(\theta^{n(\om)}\om)]/\be(\theta^{n(\om)}\om)\}\neq\emptyset$.
Then by the definition of $\tau_\om$ we obtain $0\in\tau_\om^{n(\om)+1}(I)$. 
Since $0$ is a common fixed point of all $\tau_\om$ and the image of $[0,1/\be(\om))$ by $\tau_\om$ is $[0,1)$ if $\be_\om>1$, together with the fact that $\be_\om^{(n)}\to\infty$ as $n\to\infty$, there is a positive integer $k(\om)$ such that $\tau_\om^{n(\om)+k(\om)}(I)=[0,1)$. 
This yields $(\mP\times l)(R^n(\Om\times I))\to 1$ as $n\to \infty$, which shows the support of $h\in\{h_i\}_{i=1}^N$ has full measure $1$. Since $\{h_i\}_{i=1}^N$ satisfies $h_i h_j=0$ for $i\neq j$, we obtain that $N=1$.


Let us denote by $\mu=\mP\times\hat{\mu}$ the unique $R$-invariant probability measure absolutely continuous with respect to $\mP\times l$.
We call $\hat{\mu}$ the stationary measure of the random dynamical system $R$. In the following, for $n\geq1$ we put $\tau_{\om_1\dots\om_n}=\tau_{\om_n}\circ\cdots\circ\tau_{\om_1}$ for $\om_i\in\hat{\Om}$ for $1\leq i\leq n$ and 
denote by $\hat{\mP}^n$ the product of $n$-copies of the measure $\hat{\mP}$ defined on $(\hat{\Om}^n, \cF^n)$.
The main result of this section is the following explicit formula for the  density function of the stationary measure $\hat{\mu}$:

\begin{theorem}\label{Thm A}
Assume that 
\[\int_\Om\frac{d\mP(\om)}{\be(\om)}=\int_{\hat{\Om}}\frac{d\hat{\mP}(a)}{\hat{\be}(a)}<1.\]
Set 
\[\phi(x)=\1_{[0,1]}(x)+\sum_{n=1}^\infty \int_{\hOm^n}\frac{\1_{[0,\tau_{\om_1\dots\om_n}(1)]}(x)}{\prod_{i=1}^n\hbe(\om_i)}d\hmP^n(\om_1\dots\om_n).\]

\noindent
Then $\phi$ is in $BV$ and satisfies $\cL\phi=\phi$. That is, the normalized function $\phi/||\phi||_1$ is the density function of the stationary measure $\hat{\mu}$.

\end{theorem}

\begin{proof}
Since $\int_{\hat{\Om}}d\hat{\mP}(a)/\hat{\be}(a)<1$ and $||\1_{[0,a]}||_{BV}\leq2$ for $a\in[0,1]$, we have 
\begin{linenomath}
\begin{align*}
||\phi||_{BV}
&\leq2\Biggl(1+\sum_{n=1}^\infty \int_{\hOm^n}\frac{1}{\prod_{i=1}^n\hbe(\om_i)}\hmP^n(\om_1\dots\om_n)\Biggr) \\
&\leq2\Biggl(1+\sum_{n=1}^\infty\Biggl(\int_{\hOm}\frac{d \hmP(a)}{\hbe(a)}\Biggr)^n \Biggr)<+\infty.
\end{align*}
\end{linenomath}
Then $\phi\in BV$. 

We show that $\phi$ satisfies the equation $\cL\phi=\phi$. Since $\cL$
is linear and continuous, by Lemma \ref{key}, we have
\begin{linenomath}
\begin{align*}
\cL \phi 
&=\int_\Om \cL_\om\phi\ d\mP(\om) =\int_{\hOm} \cL_{\om_0}\phi\ d\hmP(\om_0) \\
&=\int_{\hOm}\cL_{\om_0}\Biggl(\1_{[0,1]}+\sum_{n=1}^\infty \int_{\hOm^n}\frac{\1_{[0,\tau_{\om_1\dots\om_n}(1)]}}{\prod_{i=1}^n\hbe(\om_i)}d\hmP^n(\om_1\dots\om_n)\Biggr) d\hmP(\om_0) \\
&=\int_{\hOm}\Biggl(\frac{\1_{[0,\tau_{\om_0}(1)]}+[\hbe(\om_0)]}{\hbe(\om_0)} \\
&+\sum_{n=1}^\infty\int_{\hOm^n}\frac{\1_{[0,\tau_{\om_1\dots\om_n\om_0}(1)]}+[\hbe(\om_0)\tau_{\om_1\dots\om_n}(1)]}{\hbe(\om_0)\prod_{i=1}^n\hbe(\om_i)}
d\hmP^n(\om_1\dots\om_n)
\Biggr) d\hmP(\om_0) \\
&=\Biggl(\int_{\hOm}\frac{[\hbe(\om_0)]}{\hbe(\om_0)} d\hmP(\om_0)
+\sum_{n=1}^\infty \int_{\hOm^{n+1}}\frac{[\hbe(\om_0)\tau_{\om_1\dots\om_n}(1)]}{\prod_{i=0}^n\hbe(\om_i)}d\hmP^n(\om_0\dots\om_n) \Biggr)\1_{[0,1]}\\
&+\sum_{n=1}^\infty \int_{\hOm^n}\frac{\1_{[0,\tau_{\om_1\dots\om_n}(1)]}(x)}{\prod_{i=1}^n\hbe(\om_i)}d\hmP^n(\om_1\dots\om_n).
\end{align*}
\end{linenomath}
Since 
\[\cL\phi-\phi=\Biggl(\int_{\hOm}\frac{[\hbe(\om_0)]}{\hbe(\om_0)} d\hmP(\om_0)
+\sum_{n=1}^\infty \int_{\hOm^{n+1}}\frac{[\hbe(\om_0)\tau_{\om_1\dots\om_n}(1)]}{\prod_{i=0}^n\hbe(\om_i)}d\hmP^n(\om_0\dots\om_n)-1\Biggr)\1_{[0,1]},\]
what we need to show is that 
\[\int_{\hOm}\frac{[\hbe(\om_0)]}{\hbe(\om_0)} d\hmP(\om_0)
+\sum_{n=1}^\infty \int_{\hOm^{n+1}}\frac{[\hbe(\om_0)\tau_{\om_1\dots\om_n}(1)]}{\prod_{i=0}^n\hbe(\om_i)}d\hmP^n(\om_0\dots\om_n)=1.\]

We note that, by the convexity of the function $f(x)=\log x$ for $x\in (0,+\infty)$, the assumption $\int_{\hat{\Om}}d\hat{\mP}(a)/\hat{\be}(a)<1$ implies the condition ($\ref{mean expanding}$) since
\[\int_{\Om} \log \frac{1}{\be(\om)}\mP(\om)\leq \log\Bigr(\int_\Om\frac{d\mP(\om)}{\be(\om)}\Bigl)<0.\]
Then by Lemma \ref{expansion}, we know that 
\[1=\sum_{n=1}^\infty\frac{d_n(\om,1)}{\be^{(n)}_\om}\]
for $\mP$-a.e. $\om\in\Om$. Hence
\begin{linenomath}
\begin{align*}
1&=\int_{\Om}\sum_{n=1}^\infty\frac{d_n(\om,1)}{\be^{(n)}_\om}d\mP(\om)
=\int_{\Om}\frac{[\be_\om]}{\be_\om}d \mP(\om)+\sum_{n=2}^\infty\int_{\Om}\frac{[\be_{\sh^{n-1}\om}\tau_{\om}^{n-1}(1)]}{\be^{(n)}_\om}d\mP(\om) \\
&=\int_{\hOm}\frac{[\hbe(\om_0)]}{\hbe(\om_0)} d\hmP(\om_0)
+\sum_{n=1}^\infty \int_{\hOm^{n+1}}\frac{[\hbe(\om_0)\tau_{\om_1\dots\om_n}(1)]}{\prod_{i=0}^n\hbe(\om_i)}d\hmP^{n}(\om_0\dots\om_n), 
\end{align*}
\end{linenomath}
which yields the conclusion.

\end{proof}

\begin{corollary}\label{cor1}
Under the same assumption as in Theorem \ref{Thm A}, 
we have
\[\sup_{x\in[0,1]}\hat{h}(x)=\lim_{x\searrow 0}\hat{h}(x)=\frac{1}{\Bigl(1-\int_\Om\frac{d\mP(\om)}{\be(\om)}\Bigr)\Bigl(1+\sum_{n=1}^\infty\int_\Om\frac{\tau_\om^{n}(1)}{\be^{(n)}_\om}d\mP\Bigr)}\]
and 
\[\inf_{x\in[0,1]}\hat{h}(x)=\lim_{x\nearrow 1}\hat{h}(x)=\frac{1}{1+\sum_{n=1}^\infty\int_\Om\frac{\tau_\om^{n}(1)}{\be^{(n)}_\om}d\mP},\]
where $\hat{h}$ denotes the density function of the stationary measure $\hat{\mu}$.
 
\end{corollary}

\begin{proof}
Note that the indicator function $\1_I(x)$ for an open interval $I$ including $0$ is non-increasing for $x\in[0,1]$. Hence
\[\phi(x)=\1_{[0,1]}(x)+\sum_{n=1}^\infty \int_{\hOm^n}\frac{\1_{[0,\tau_{\om_1\dots\om_n}(1)]}(x)}{\prod_{i=1}^n\hbe(\om_i)}d\hmP^n(\om_1\dots\om_n), \]
it is also non-increasing for $x\in[0,1]$. Then we have
\begin{linenomath}
\begin{align*}
\sup_{x\in[0,1]}\phi(x)
&=\lim_{x\searrow 0}\phi(x) \\
&=1+\sum_{n=1}^\infty\int_{\hOm^n}\frac{d\hmP^n(\om_1\dots\om_n)}{\prod_{i=1}^n\hbe(\om_i)} \\
&=1+\sum_{n=1}^\infty\Biggl(\int_{\hOm}\frac{d\hmP(a)}{\hbe(a)}\Biggr)^n \\
&=\frac{1}{1-\int_\Om\frac{d\mP(\om)}{\be(\om)}}
\end{align*}
\end{linenomath}
and 
\begin{linenomath}
\begin{align*}
\inf_{x\in[0,1]}\phi(x)
&=\lim_{x\nearrow 1}\phi(x) \\
&=1+\sum_{n=1}^\infty\int_{\hOm^n}\frac{\lim_{x\nearrow1}\1_{[0,\tau_{\om_1\dots\om_n}(1)]}(x)}{\prod_{i=1}^n\hbe(\om_i)}d\hmP^n(\om_1\dots\om_n) \\
&=1.
\end{align*}
\end{linenomath}
Together with Theorem \ref{Thm A}, we have the conclusion.
\end{proof}

\end{section}

\begin{section}{Linear response formula for random dynamics generated by finitely many beta-transformations}

In this section, we investigate differentiability of the density function of the stationary measure $\hat{\mu}$ when the probability $\mP$ varies. Such a problem, called linear response in the context of thermodynamic formalism, has been investigated for i.i.d random dynamical systems generated by fully branched expanding maps in \cite{Ba-Ru} (see \cite{Ga-Se} and \cite{Se-Ru} for samplewise linear response).

Under the same assumptions as in the previous section, we have the explicit formula for the density function of the stationary measure $\hat{\mu}$ given by Theorem \ref{Thm A}. As its application, we give a linear response formula for the density function explicitly when a random dynamics is generated by finitely many beta-transformations and $\mP$ is a Bernoulli measure. In detail, we assume that there is a probability space $(\Om_0,\cF_0,\mP_0)$ such that $(\Om,\cF,\mP)=(\Om_0,\cF_0,\mP_0)^{\N}$, where 
$\Om_0$ is a finite set $\{0,\dots,N-1\}$, $\cF_0$ is the set of all subset of $\Om_0$ and 
$\mP_0$ is a probability measure on the finite state $\Om_0$ given by $\mP_0(\{i\})=p_i$ for $0\leq i\leq N-1$ with $p_i>0$ and $p_0+\dots+p_{N-1}=1$.
Note that the random variable $\hbe$ takes only at most $N$-values and we set 
$\be_i=\hbe(i)$ for $0\leq i\leq N-1$. In this setting, the explicit formula for the density function of $\hat{\mu}$ given in Theorem \ref{Thm A} becomes simpler as follows. 

\begin{theorem} 
Let $N\geq2$ and let $(\Om,\cF,\mP)=(\Om_0,\cF_0,\mP_0)^{\N}$ be a one-sided product space defined as above. Under the assumption that 
\[\sum_{i=0}^{N-1}p_i/\be_i<1,\]
the function $\phi$ defined as in Theorem \ref{Thm A} is given by
\[\phi(x)=\1_{[0,1]}(x)+\sum_{n=1}^\infty\sum_{\om_1\dots\om_n\in\Om_0^n}\prod_{i=0}^{N-1}\Biggl(\frac{p_i}{\be_i}\Biggr)^{W_i(\om_1\dots\om_n)}\1_{[0,\tau_{\om_1\dots\om_n}(1)]},
\]
where $W_i(\om_1\dots\om_n)$ denotes the number of the set $\{1\leq j\leq n\ ;\ \om_j=i\}$ for $0\leq i\leq N-1$. 
\end{theorem}

\begin{proof}
The statement follows from the fact that 
\begin{linenomath}
\begin{align*}
&\int_{\Om_0^n}\frac{\1_{[0,\tau_{\om_1\dots\om_n}(1)]}}{\hbe(\om_1)\dots\hbe(\om_n)}d\hmP^n(\om_1\dots\om_n) \\
&=\sum_{v_1\dots v_n  \in\Om_0^n} \frac{\hmP^n(\{v_1\dots v_n\})}{\be_0^{W_0(v_1\dots v_n)}\cdots\be_{N-1}^{W_{N-1}(v_1\dots v_n)}}\1_{[0,\tau_{v_1\dots v_n}(1)]} \\
&=\sum_{v_1\dots v_n  \in\Om_0^n} \frac{p_0^{W_0(v_1\dots v_n)}\cdots p_{N-1}^{W_{N-1}(v_1\dots v_n)}}{\be_0^{W_0(v_1\dots v_n)}\cdots\be_{N-1}^{W_{N-1}(v_1\dots v_n)}}\1_{[0,\tau_{v_1\dots v_n}(1)]}
\end{align*}
\end{linenomath}
for $n\geq1$.
\end{proof}

In the rest of this section we consider the case of $N=2$ for simplicity. We assume that $\be_1>1$ and $\be_0\leq\be_1$. 
Note that if $\be_0\geq1$ then $p/\be_1+(1-p)/\be_0<1$ for all $p\in[0,1]$. In that case, we set $p_c=0$ and if $\be_0<1$
we set $p_c\in(0,1)$ as the unique solution of the equation
\[\frac{p}{\be_1}+\frac{1-p}{\be_0}=1\]
for $p\in[0,1]$. By putting $p=p_1\in(p_c,1)$, the function 
$\phi$ defined in Theorem \ref{Thm A} is given by
\[\phi(p)=\1_{[0,1]}+\sum_{n=1}^\infty\sum_{\om_1\dots\om_n\in\{0,1\}^n}\Biggl(\frac{p}{\be_1}\Biggr)^{\sum_{i=1}^n\om_i}\Biggl(\frac{1-p}{\be_0}\Biggr)^{n-\sum_{i=1}^n\om_i}\1_{[0,\tau_{\om_1\dots\om_n}(1)]}.\]
The main theorem of this section is the following:

\begin{theorem}\label{Thm B}

(1) The function $p\in(p_c,1)\mapsto \phi(p)\in BV$ is analytic and its derivative is given by
\begin{linenomath}
\begin{align*}
\frac{\partial \phi(p)}{\partial p}
&=\sum_{n=1}^\infty\sum_{\substack{\om_1\dots\om_n\in\{0,1\}^n \\ \om_1\dots\om_n\neq 0^n,1^n}}
p^{\sum_{i=1}^n\om_i-1}(1-p)^{n-\sum_{i=1}^n\om_i-1}
\Bigl(\sum_{i=1}^n\om_i-np\Bigr)
\frac{\1_{[0,\tau_{\om_1\dots\om_n}(1)]}}{\be_{\om_1\dots\om_n}} \\
&+\sum_{n=1}^\infty n p^{n-1}\frac{\1_{[0,T_{\be_1}^n(1)]}}{\be_1^n}-\sum_{n=1}^\infty n (1-p)^{n-1}\frac{\1_{[0,T_{\be_0}^n(1)]}}{\be_0^n}.
\end{align*}
\end{linenomath}

(2) The function $p\in(p_c,1)\mapsto \int\phi_p dl\in (0,+\infty)$ is analytic and its derivative is given by
\begin{linenomath}
\begin{align*}
\frac{\partial \int\phi(p) dl}{\partial p}
&=\sum_{n=1}^\infty\sum_{\substack{\om_1\dots\om_n\in\{0,1\}^n \\ \om_1\dots\om_n\neq 0^n,1^n}}
p^{\sum_{i=1}^n\om_i-1}(1-p)^{n-\sum_{i=1}^n\om_i-1}
\Bigl(\sum_{i=1}^n\om_i-np\Bigr)
\frac{\tau_{\om_1\dots\om_n}(1)}{\be_{\om_1\dots\om_n}} \\
&+\sum_{n=1}^\infty n p^{n-1}\frac{T_{\be_1}^n(1)}{\be_1^n}-\sum_{n=1}^\infty n (1-p)^{n-1}\frac{T_{\be_0}^n(1)}{\be_0^n}.
\end{align*}
\end{linenomath}

(3) Let $h(p)=\phi(p)/\int\phi(p)dl$. Then the function $p\in(p_c,1)\mapsto h(p) \in BV$ is analytic and its derivative is given by
\[\frac{\partial h(p)}{\partial p}=\frac{\frac{\partial \phi(p)}{\partial p}\int\phi(p) dl-\phi(p)\frac{\partial \int\phi(p) dl}{\partial p}}{\Bigl(\int\phi(p) dl\Bigr)^2},\]
where all $\phi(p)$, $\displaystyle{\int\phi(p) dl}$, $\displaystyle{\frac{\partial \phi(p)}{\partial p}}$ and $\displaystyle{\frac{\partial \int\phi(p) dl}{\partial p}}$ are given explicitly.

\end{theorem}

\begin{proof}
Let $p_0\in(p_c,1)$ and set $I=[p_0,1]$. Define the formal power series $\rho(z)$ by
\[\rho(z)=\1_{[0,1]}+\sum_{n=1}^\infty\sum_{\om_1\dots\om_n\in\{0,1\}^n}\Biggl(\frac{z}{\be_1}\Biggr)^{\sum_{i=1}^n\om_i}\Biggl(\frac{1-z}{\be_0}\Biggr)^{n-\sum_{i=1}^n\om_i}\1_{[0,\tau_{\om_1\dots\om_n}(1)]}.\]
Note that we formally have $\rho(p)=\phi(p)$ for $p\in I$.

(1) To show the analyticity of the function $p\mapsto\phi(p)$ at $p_0$, we see that the formal power series $\rho(z)$ absolutely converges in some neighborhood 
$ D\subset \C$ of $I$. Note that for any $n\geq1$ and $\om_1\dots\om_n\in\{0,1\}^n$, it holds that 
\[||\1_{[0, \tau_{\om_1\dots\om_n}(1)]}||_{BV}\leq1+\tau_{\om_1\dots\om_n}(1)\leq2.\] 
Hence
\[\begin{split}
||\phi(z)||_{BV}
&\leq2+2\sum_{n=1}^\infty\sum_{\om_1\dots\om_n\in\{0,1\}^n}\Bigl|\frac{z}{\be_1}\Bigr|^{\sum_{i=1}^n\om_i}\Bigl|\frac{1-z}{\be_0}\Bigr|^{n-\sum_{i=1}^{n}\om_i} \\
&\leq2+2\sum_{n=1}^\infty\Biggl(\Bigl|\frac{z}{\be_1}\Bigr|+\Bigl|\frac{1-z}{\be_0}\Bigr|\Biggr)^n \\ 
&=2+\frac{2}{1-\Bigl(\frac{|z|}{\be_1}+\frac{|1-z|}{\be_0}\Bigr)}<+\infty
\end{split}
\]
whenever $z\in\C$ satisfies $|z|/\be_1+|1-z|/\be_0<1$. 
Since $I$ is included in the open region $\{z\in\C\ ;\ |z|/\be_1+|1-z|/\be_0<1\}$, we can take an open neighborhood $D$ of $I$ such that $\phi(z)$ converges uniformly on $D\subset\C$, which yields the analyticity of the function $p\mapsto \phi(p)$. 

For $p\in(p_c,1)$, set $\del=(p+|\ep|)/\be_1+(1-p+|\ep|)/\be_0$, where $\ep$ is taken so small that $\del<1$. 
Since
\begin{linenomath}
\begin{align*}
&\Biggl|\frac{1}{\ep}\Biggl\{\Bigl(\frac{p+\ep}{\be_1}\Bigr)^k\Bigl(\frac{1-(p+\ep)}{\be_0}\Bigr)^{n-k}
-\Bigl(\frac{p}{\be_1}\Bigr)^k\Bigl(\frac{1-p}{\be_0}\Bigr)^{n-k}\Biggr\} \Biggr| \\
&\leq \Biggl|\frac{(p+\ep)^k-p^k}{\ep\be_1^k}\Biggr|\Biggl| \frac{1-(p+\ep)}{\be_0}\Biggr|^{n-k}
+\Bigl(\frac{p}{\be_1}\Bigr)^k\Biggl|\frac{(1-(p+\ep))^{n-k}-(1-p)^{n-k}}{\ep\be_0^{n-k}}\Biggr|  \\
&\leq k\cdot\frac{(p+|\ep|)^{k-1}}{\be_1^{k-1}}\cdot\Biggl| \frac{1-p+|\ep|}{\be_0}\Biggr|^{n-k}
+(n-k)\cdot\Bigl(\frac{p}{\be_1}\Bigr)^k\cdot\Bigl(\frac{1-p+|\ep|}{\be_0}\Bigr)^{n-k-1}
\end{align*}
\end{linenomath}
for $n\geq1$ and $1\leq k <n$, we have
\begin{linenomath}
\begin{align*}
&\delta_1(n):=\Biggl|\Biggl|\frac{1}{\ep}\sum_{\substack{\om_1\dots\om_n\in\{0,1\}^n \\ \om_1\dots\om_n\neq 0^n,1^n}}  
\Biggl\{\Bigl(\frac{p+\ep}{\be_1}\Bigr)^{\sum_{i=1}^n\om_i}\Bigl(\frac{1-(p+\ep)}{\be_0}\Bigr)^{n-\sum_{i=1}^n\om_i} \\
&-\Bigl(\frac{p}{\be_1}\Bigr)^{\sum_{i=1}^n\om_i}\Bigl(\frac{1-p}{\be_0}\Bigr)^{n-\sum_{i=1}^n\om_i}\Biggr\}\1_{[0, \tau_{\om_1\dots\om_n}(1)]}\Biggr|\Biggr|_{BV}\\
&\leq2\sum_{\substack{\om_1\dots\om_n\in\{0,1\}^n \\ \om_1\dots\om_n\neq 0^n,1^n}}  
\Biggl\{\sum_{i=1}^n\om_i\Bigl(\frac{p+|\ep|}{\be_1}\Bigr)^{\sum_{i=1}^n\om_i-1}\Bigl(\frac{1-p+|\ep|}{\be_0}\Bigr)^{n-\sum_{i=1}^n\om_i} \\
&+(n-\sum_{i=1}^n\om_i)\Bigl(\frac{p+|\ep|}{\be_1}\Bigr)^k\cdot\Bigl(\frac{1-p+|\ep|}{\be_0}\Bigr)^{n-\sum_{i=1}^n\om_i-1} \Biggr\}\\
&\leq 2\sum_{k=1}^{n-1}\Biggl\{\binom{n}{k}k\cdot\frac{(p+|\ep|)^{k-1}}{\be_1^{k-1}}\cdot\Biggl| \frac{1-p+|\ep|}{\be_0}\Biggr|^{n-k} \\
&+\binom{n}{n-k}(n-k)\cdot\Bigl(\frac{p+|\ep|}{\be_1}\Bigr)^k\cdot\Bigl(\frac{1-p+|\ep|}{\be_0}\Bigr)^{n-k-1}\Biggr\} \\
&\leq 4n\del^{n-1}.
\end{align*}
\end{linenomath}
In addition, 
\begin{linenomath}
\begin{align*}
&\delta_2(n):=\Biggl|\Biggl|\frac{1}{\ep}\Bigl\{\Bigl(\frac{p+\ep}{\be_1}\Bigr)^n-\Bigl(\frac{p}{\be_1}\Bigr)^n\Bigr\}\1_{[0,T_{\be_1}^n(1)]}\Biggr|\Biggr|_{BV} \\
&+\Biggl|\Biggl|\frac{1}{\ep}\Bigl\{\Bigl(\frac{1-(p+\ep)}{\be_0}\Bigr)^n-\Bigl(\frac{1-p}{\be_0}\Bigr)^n\Bigr\}\1_{[0,T_{\be_0}^n(1)]}\Biggr|\Biggr|_{BV} \\
&\leq 4n\del^{n-1}
\end{align*}
\end{linenomath}
for $n\geq1$.

Let us denote by $(BV^*, ||\cdot||_{BV^*})$ the Banach space of bounded linear functionals on $BV$. By using the above inequalities for $f^*\in BV^*$ we have
\begin{linenomath}
\begin{align*} 
&\frac{1}{||f^*||_{BV^*}}\Biggl|\frac{f^*(\phi(p+\ep))-f^*(\phi(p))}{\ep} \Biggr|\\
&\leq\Biggl|\Biggl|\sum_{n=1}^\infty\sum_{\om_1\dots\om_n\in\{0,1\}^n}\Biggl\{\Biggl(\frac{p+\ep}{\be_1}\Biggr)^{\sum_{i=1}^n\om_i}\Biggl(\frac{1-(p+\ep)}{\be_0}\Biggr)^{n-\sum_{i=1}^n\om_i} \\
&-\Biggl(\frac{p}{\be_1}\Biggr)^{\sum_{i=1}^n\om_i}\Biggl(\frac{1-p}{\be_0}\Biggr)^{n-\sum_{i=1}^n\om_i}\Biggr\}\1_{[0,\tau_{\om_1\dots\om_n}(1)]}\Biggr|\Biggr|_{BV}  \\
&\leq\sum_{n=1}^\infty\Biggl|\Biggl|\frac{1}{\ep}\sum_{\substack{\om_1\dots\om_n\in\{0,1\}^n \\ \om_1\dots\om_n\neq 0^n,1^n}} \Biggl\{\Bigl(\frac{p+\ep}{\be_1}\Bigr)^{\sum_{i=1}^n\om_i}\Bigl(\frac{1-(p+\ep)}{\be_0}\Bigr)^{n-\sum_{i=1}^n\om_i} \\
&-\Bigl(\frac{p}{\be_1}\Bigr)^{\sum_{i=1}^n\om_i}\Bigl(\frac{1-p)}{\be_0}\Bigr)^{n-\sum_{i=1}^n\om_i}\Biggr\}\1_{[0, \tau_{\om_1\dots\om_n}(1)}\Biggr|\Biggr|_{BV}\\
&\leq\sum_{n=1}^\infty(\del_1(n)+\del_2(n)) \\
&\leq \sum_{n=1}^\infty 8n\del^{n-1}=\frac{8}{(1-\del)^2}<+\infty.
\end{align*}
\end{linenomath}
Note that by continuity of $f^*$ it holds that
\begin{linenomath}
\begin{align*}
&f^*(\phi(p)) \\
&=f^*(\1_{[0,1]})+\sum_{n=1}^\infty\sum_{\om_1\dots\om_n\in\{0,1\}^n}\Biggl(\frac{p}{\be_1}\Biggr)^{\sum_{i=1}^n\om_i}\Biggl(\frac{1-p}{\be_0}\Biggr)^{n-\sum_{i=1}^n\om_i}f^*(\1_{[0,\tau_{\om_1\dots\om_n}(1)]}).
\end{align*}
\end{linenomath}
Since for every $\om_1\dots\om_n\in\{0,1\}^n$ the function
\[p\in(p_c,1)\mapsto\Biggl(\frac{p}{\be_1}\Biggr)^{\sum_{i=1}^n\om_i}\Biggl(\frac{1-p}{\be_0}\Biggr)^{n-\sum_{i=1}^n\om_i}f^*(\1_{[0,\tau_{\om_1\dots\om_n}(1)]})\in\C\]
is analytic, together with the Lebesgue convergence theorem, we have
\begin{linenomath}
\begin{align*}
&\lim_{\ep\to0}\frac{f^*(p+\ep)-f^*(p)}{\ep} \\
&=\sum_{n=1}^\infty\sum_{\om_1\dots\om_n\in\{0,1\}^n}\frac{\partial}{\partial p}\Biggl(\frac{p}{\be_1}\Biggr)^{\sum_{i=1}^n\om_i}\Biggl(\frac{1-p}{\be_0}\Biggr)^{n-\sum_{i=1}^n\om_i}f^*(\1_{[0,\tau_{\om_1\dots\om_n}(1)]}) \\
&=
\sum_{n=1}^\infty\sum_{\substack{\om_1\dots\om_n\in\{0,1\}^n \\ \om_1\dots\om_n\neq 0^n,1^n}}\frac{\partial}{\partial p}\Biggl(\frac{p}{\be_1}\Biggr)^{\sum_{i=1}^n\om_i}\Biggl(\frac{1-p}{\be_0}\Biggr)^{n-\sum_{i=1}^n\om_i}f^*(\1_{[0,\tau_{\om_1\dots\om_n}(1)]}) \\
&+\sum_{n=1}^\infty \frac{\partial}{\partial p}p^n\frac{f^*(\1_{[0,T_{\be_1}^n(1)]})}{\be_1^n}+\sum_{n=1}^\infty \frac{\partial}{\partial p}(1-p)^n\frac{f^*(\1_{[0,T_{\be_0}^n(1)]})}{\be_0^n} \\
&=\sum_{n=1}^\infty\sum_{\substack{\om_1\dots\om_n\in\{0,1\}^n \\ \om_1\dots\om_n\neq 0^n,1^n}}
p^{\sum_{i=1}^n\om_i-1}(1-p)^{n-\sum_{i=1}^n\om_i-1}
\Bigl(\sum_{i=1}^n\om_i-np\Bigr)
\frac{f^*(\1_{[0,\tau_{\om_1\dots\om_n}(1)]})}{\be_{\om_1\dots\om_n}} \\
&+\sum_{n=1}^\infty n p^{n-1}\frac{f^*(\1_{[0,T_{\be_1}^n(1)]})}{\be_1^n}-\sum_{n=1}^\infty n (1-p)^{n-1}\frac{f^*(\1_{[0,T_{\be_0}^n(1)]})}{\be_0^n}.
\end{align*}
\end{linenomath}
Since $f^*$ is any linear bounded functional on $BV$, we obtain the conclusion.

(2) In the proof of (1) we see that for $p\in (p_c,1]$ and $\ep\in\R\setminus\{0\}$ whose modulus is sufficient small, the value $|| (\phi(p+\ep)-\phi(p))/\ep||_{BV}$ is bounded above by $8(1-\del)^{-2}$, which is independent of $\ep$. Since $|| (\phi(p+\ep)-\phi(p))/\ep||_{1}\leq|| (\phi(p+\ep)-\phi(p))/\ep||_{BV}$, by applying the Lebesgue convergence theorem, we obtain
\begin{linenomath}
\begin{align*}
\frac{\partial}{\partial p}\int\phi(p) dl=\int\frac{\partial}{\partial p}\phi(p) dl.
\end{align*}
\end{linenomath}
Note that the quantity $\del_1(n)$ and $\del_2(n)$ is summable for $n\geq1$ as we see in the proof of (1). Hence 
\begin{linenomath}
\begin{align*}
&\int\frac{\partial}{\partial p}\phi(p) dl \\
&=\sum_{n=1}^\infty\sum_{\substack{\om_1\dots\om_n\in\{0,1\}^n \\ \om_1\dots\om_n\neq 0^n,1^n}}
p^{\sum_{i=1}^n\om_i-1}(1-p)^{n-\sum_{i=1}^n\om_i-1}
\Bigl(\sum_{i=1}^n\om_i-np\Bigr)
\frac{\int\1_{[0,\tau_{\om_1\dots\om_n}(1)]}dl}{\be_{\om_1\dots\om_n}} \\
&+\sum_{n=1}^\infty n p^{n-1}\frac{\int\1_{[0,T_{\be_1}^n(1)]}dl}{\be_1^n}-\sum_{n=1}^\infty n (1-p)^{n-1}\frac{\int\1_{[0,T_{\be_0}^n(1)]}dl}{\be_0^n},
\end{align*}
\end{linenomath}
which yields the conclusion.

(3) immediately follows from (1) and (2). 
\end{proof}

\end{section}

\begin{section}{Invariant densities for non-i.i.d. cases}
In this section we consider non-i.i.d. random dynamical systems of beta transformations. 
We assume that the noise dynamics $\sh:\Om\to\Om$ is ergodic and invertible, and the measurable function $\be:\Om\to(0,+\infty)$ satisfies $\be(\om)\in(1,+\infty)$ for $\mP$-a.e. $\om\in\Om$. A key ingredient to construct a unique $R$-invariant probability measure $\mu$ absolutely continuous with respect to $\mP\times l$ is that for $h\in L^1(\mP\times l)$ with $h\geq0$ we have that $\cL_\om h_\om=h_{\sh\om}$ for $\mP$-a.e. $\om\in\Om$, where $h_\om=h(\om, \cdot)$ for $\om$, if and only if the finite positive measure given by
\[m_h(A\times J):=\int_A\int_J h(\om,x) d(\mP\times l)\ \ (A\times J\in \cF\times\cB)\]
is $R$-invariant (see \cite{Bu}). As stated in Example 0.4 in \cite{Bu}, by applying Theorem 0.3 in \cite{Bu} to our setting, 
we can see that there is a unique $R$-invariant probability measure $\mu$ absolutely continuous with respect to $\mP\times l$. 
We construct a function $\vp\in L^1(\mP\times l)$ satisfying $\cL \vp_\om=\vp_{\sh\om}$ for $\mP$-a.e. $\om\in\Om$ in the following two cases. 

\begin{subsection}{Strongly expanding cases}
First we give an explicit formula for the density function of $\mu$ in the case where  the maps randomly chosen are strongly expanding in the sense that ${\displaystyle{\rm essinf}_{\om\in\Om} \be(\om)>2}$. 
\begin{theorem}\label{Thm C}
Under the assumption ${\displaystyle{\rm essinf}_{\om\in\Om} \be(\om)>2}$, the measurable function $c:\Om\to\R$ given by
\[c(\om)=\sum_{n=0}^\infty (-1)^n c_n(\om)\]
for $\om\in\Om$, where $c_0(\om)=1$,
\[c_1(\om)=\sum_{i_1=1}^\infty\frac{\tau_{\sh^{-i_1}\om}^{i_1}(1)}{\be_{\sh^{-i_1}\om}^{(i_1)}}\]
and
\[c_n(\om)=\sum_{i_1=1}^\infty\frac{\tau_{\sh^{-i_1}\om}^{i_1}(1)}{\be_{\sh^{-i_1}\om}^{(i_1)}}\sum_{i_2=1}^\infty\frac{\tau_{\sh^{-(i_1+i_2)}\om}^{i_2}(1)}{\be_{\sh^{-(i_1+i_2)}\om}^{(i_2)}}\cdots\sum_{i_n=1}^\infty\frac{\tau_{\sh^{-(i_1+\cdots+i_n)}\om}^{i_n}(1)}{\be_{\sh^{-(i_1+\cdots+i_n)}\om}^{(i_n)}}\]
for $n\geq2$, is a positive function in $L^\infty(\mP)$.
Furthermore, the function $\varphi:\Om\times[0,1]\to\R$ given by
\[\varphi(\om,x)=\sum_{n=0}^\infty\frac{c(\sh^{-n}\om)\1_{[0, \tau_{\sh^{-n}\om}^{n}(1)]}(x)}{\be^{(n)}_{\sh^{-n}\om}}\]
for $(\om,x)\in\Om\times[0,1]$ is non-negative and in $L^1(\mP\times l)$ with $\displaystyle{\int\varphi d\mP\times l}=1$ satisfying
\begin{equation}\label{PF}
\cL_\om \vp_\om=\varphi_{\sh\om}
\end{equation}
for $\mP$-a.e. $\om\in\Om$. That is, $\vp$ is the density function of the $R$-invariant probability measure $\mu$.

\end{theorem}
For the proof of Theorem \ref{Thm C} we need the following two lemmas. 
Let us denoted by $U_\sh$ the Koopman operator associated with $\sh$ on $L^1(\mP)$, i.e., $U_{\sh}f=f\circ\sh$ for $f\in L^1(\mP)$. We also denote by $U_{\sh^{-1}}$ the Koopman operator associated with $\sh^{-1}$. In the following, by slight abuse of notation, we denote by $||\cdot||$ the $L^1$ norm on $L^1(\mP)$ and the operator norm for bounded linear operators on $L^1(\mP)$.

\begin{lemma}\label{lemma C-1}
Assume that $\ga:={\displaystyle{\rm essinf}_{\om\in\Om}\be(\om)>1}$. Let $T:L^1(\mP)\to L^1(\mP)$ be the operator defined by
\[Tf=\sum_{n=1}^\infty\frac{d_n(\sh^{-n}\om,1)}{\be_{\sh^{-n}\om}^{(n)}}U_{\sh^{-n}}f\]
for $f\in L^1(\mP)$ and let $S:L^1(\mP)\to L^1(\mP)$ be the operator defined by 
\[Sf=\sum_{n=1}^\infty\frac{\tau_{\sh^{-n}\om}^{n}(1)}{\be_{\sh^{-n}\om}^{(n)}}U_{\sh^{-n}}f\]
for $f\in L^1(\mP)$. Then $T$ and $S$ are linear and bounded on $L^1(\mP)$ satisfying
\[I-T=(I-U_{\sh^{-1}})(I+S),\]
where $I$ denotes the identity map on $L^1(\mP)$. 
\end{lemma}

\begin{proof}
Since 
\begin{align*}
||Tf||
&\leq \int_\Om \sum_{n=1}^\infty \frac{d_n(\sh^{-n}\om,1)}{\be_{\sh^{-n}\om}^{(n)}}|f(\sh^{-n}\om)|d\mP(\om)
=  \int_\Om \sum_{n=1}^\infty \frac{d_n(\om,1)}{\be_{\om}^{(n)}}|f(\om)| d\mP(\om) \\
&=||f||
\end{align*}
and
\begin{align*}
||Sf||
&\leq \int_\Om \sum_{n=1}^\infty \frac{\tau_{\sh^{-n}\om}^n(1)}{\be_{\sh^{-n}\om}^{(n)}}|f(\sh^{-n}\om)|d\mP(\om)
=  \int_\Om \sum_{n=1}^\infty \frac{\tau_\om^n(1)}{\be_{\om}^{(n)}}|f(\om)| d\mP(\om) \\
&\leq \frac{\ga}{\ga-1}||f||
\end{align*}
for $f\in L^1(\mP)$, we have that the operators $T$ and $S$ are bounded on $L^1(\mP)$. The fact that these operators are linear follows from their definitions directly.

For each positive integer $n\geq1$ let us define the operators $P_n$ and $Q_n$ by
\[
(P_nf)(\om)=\frac{d_n(\om,1)}{\be_\om^{(n)}}f(\om)
\]
and 
\[(Q_nf)(\om)=\frac{\tau_\om^n(1)}{\be_\om^{(n)}}f(\om)\]
for $f \in L^1(\mP)$ and $\om\in\Om$. Since the functions $d_n(\om,1)/\be_\om^{(n)}$ and $\tau_\om^n(1)/\be_\om^{(n)}$ are in $L^\infty(\mP)$, we know that $P_n$ and $Q_n$ are linear bounded operators on $L^1(\mP)$. From the definition of the map $\tau$ and the digit function $d_n$ for each $n\geq1$ we have
\[\frac{d_n(\om,1)}{\be_\om^{(n)}}=\frac{\tau_\om^{n-1}(1)}{\be_\om^{(n-1)}}-\frac{\tau_\om^{n}(1)}{\be_\om^{(n)}}\]
for $\om\in\Om$ and $n\geq1$, which yields the equation
\[P_n=Q_{n-1}-Q_n\]
for each $n\geq1$, where $Q_0=I$. 
Hence
\begin{linenomath}
\begin{align*}
(I-T)
&=I-\sum_{n=1}^\infty U_{\sh^{-n}}P_n
=I-\sum_{n=1}^\infty U_{\sh^{-n}}(Q_{n-1}-Q_n) \\
&=S-U_{\sh^{-1}}\sum_{n=1}^\infty U_{\sh^{-(n-1)}}Q_{n-1}
=(I-U_{\sh^{-1}})(I+S),
\end{align*}
\end{linenomath}
which yields the conclusion.
\end{proof}

\begin{lemma}\label{lemma C-2}
Assume that $\ga:={\displaystyle{\rm essinf}_{\om\in\Om}\be(\om)>2}$. 
Let $S$ be the linear bounded operator on $L^1(\mP)$ defined as in Lemma \ref{lemma C-1}.
Then we have $||S||<1$. In particular, $I+S$ has the inverse given by 
\[(I+S)^{-1}=\sum_{n=0}^{\infty}(-S)^n.\]
\end{lemma}

\begin{proof}
By the fact that $\tau_\om^n(1)\in[0,1]$ for any $n\geq0$ and $\om\in\Om$, we have
\[||Sf||
\leq\sum_{n=1}^\infty\frac{1}{\ga^n}||f||
\leq\frac{1}{\ga-1}||f||.\]
Since $1/(\ga-1)<1$ by the assumption $\ga>2$, we obtain $||S||<1$ by the above inequality. This yields the Neumann series $\sum_{n=0}^\infty(-S)^n$ absolutely converges and it is the inverse of $(I+S)$:
\[(I+S)^{-1}=\sum_{n=0}^\infty(-S)^n.\]

\end{proof}

\begin{proof}[Proof of Theorem \ref{Thm C}]
Let $n\geq1$. By definition the function $c_n(\om)$ is non-negative and 
\[c_n(\om)\leq\Biggl(\sum_{i=1}\frac{1}{\ga^i}\Biggr)^n=\Biggl (\frac{1}{\ga-1}\Biggr)^n.
\]
Since $\ga>2$, we have that $(\ga-1)^{-1}<1$, which guarantees that 
\[|c(\om)|
\leq\frac{1}{1-\frac{1}{\ga-1}}
=\frac{\ga-1}{\ga-2}<+\infty\]
for $\mP$-a.e. $\om\in\Om$. This shows that $c(\om)$ is in $L^\infty(\mP)$. 

By Lemma \ref{key}, we see that
\begin{linenomath}
\begin{align*}
\cL_\om\vp_\om
&=\sum_{n=0}^\infty\frac{c(\sh^{-n}\om)}{\be_{\sh^{-n}\om}^{(n)}}
\Biggr(\frac{[\be_\om\tau_{\sh^{-n}\om}^n(1)]}{\be_\om}+\frac{\1_{[0,\tau_\om\circ\tau_{\sh^{-n}\om}^n(1)]}}{\be_\om}\Biggr) \\
&=\sum_{n=0}^\infty\frac{c(\sh^{-n}\om)}{\be_{\sh^{-n}\om}^{(n+1)}}[\be_\om\tau_{\sh^{-n}\om}^n(1)]\1_{[0,1]}
+\sum_{n=0}^\infty\frac{c(\sh^{-n}\om)}{\be_{\sh^{-n}\om}^{(n+1)}}\1_{[0,\tau_{\sh^{-n}\om}^{n+1}(1)]} \\
&=\sum_{n=0}^\infty\frac{c(\sh^{-n}\om)}{\be_{\sh^{-n}\om}^{(n+1)}}d_{n+1}(\sh^{-n}\om,1) \1_{[0,1]}
+\sum_{n=1}^\infty\frac{c(\sh^{-n}(\sh\om))}{\be_{\sh^{-n}(\sh\om)}^{(n)}}\1_{[0,\tau_{\sh^{-n}(\sh\om)}^{n}(1)]}.
\end{align*}
\end{linenomath}
Since 
\[\cL_\om\vp_\om-\vp_\om=\Biggl(\sum_{n=0}^\infty\frac{c(\sh^{-n}\om)}{\be_{\sh^{-n}\om}^{(n+1)}}d_{n+1}(\sh^{-n}\om,1)-c(\sh\om)\Biggr)\1_{[0,1]}, \]
we have $\cL_\om\vp_\om=\vp_{\sh\om}$ if $c(\om)$ satisfies 
\begin{equation}\label{functional equation}
c(\om)=\sum_{n=1}^\infty\frac{d_{n}(\sh^{-n}\om,1)}{\be_{\sh^{-n}\om}^{(n)}}c(\sh^{-n}\om)
\end{equation}
for $\mP$-a.e. $\om\in\Om$.

As we see in Lemma \ref{lemma C-1}, the operators $T:L^1(\mP)\to L^1(\mP)$ and $S:L^1(\mP)\to L^1(\mP)$ defined by 
\[Tf=\sum_{n=1}^\infty\frac{d_n(\sh^{-n}\om,1)}{\be_{\sh^{-n}\om}^{(n)}}U_{\sh^{-n}}f, \ 
Sf=\sum_{n=1}^\infty\frac{\tau_{\sh^{-n}\om}^{n}(1)}{\be_{\sh^{-n}\om}^{(n)}}U_{\sh^{-n}}f\]
for $f\in L^1(\mP)$ are linear and bounded on $L^1(\mP)$ satisfying 
\[I-T=(I-U_{\sh^{-1}})(I+S).\]
We note that the equation (\ref{functional equation}) is equivalent to 
\[(I-T)c(\om)=0\]
for $\mP$-a.e. $\om\in\Om$. That is, if the function $c(\om)$ is in the kernel of $I-T$ we obtain the desired result. By Lemma \ref{lemma C-2}, the operator $I+S$ has the inverse $\sum_{n=1}^\infty(-S)^n$. By the definition of the operator $S$ and the function $c(\om)$ we see that 
\[(I+S)^{-1}1=\sum_{n=0}^{\infty}(-S)^n 1=\sum_{n=0}^{\infty}(-1)^n c_n(\om)=c(\om).\]
Hence we have 
\[(I-T)c=(I-U_{\sh^{-1}})(I+S)(I+S)^{-1}1=(I-U_{\sh^{-1}})1=0.\]
This shows that the function $c(\om)$ is in the kernel of $I-T$, which gives the conclusion.
\end{proof}

\end{subsection}

\begin{subsection}{Small random perturbation of the map for a non-simple number}
In the following we consider the case that the value of the random valuable $\be(\om)$ is sufficiently close to some non-simple number. Recall that $\be_0>1$ is said to be  non-simple if there is no positive integer $N$ such that 
$\tau_{\be_0}^N(1)=0$. Put $\ep'=\min\{\langle\be_0\rangle, 1-\langle\be_0\rangle\}$. Then every $\be\in(\be_0-\ep', \be_0+\ep')$ has the same integer part, that is, $[\be]=[\be_0]$. 

The important fact to show the main result is that if the image of $\be(\om)$ is in $(\be_0-\ep, \be_0+\ep)$ for sufficiently small $\ep>0$ then the operator $I+S$, where $S :L^1(\mP)\to L^1(\mP)$ is the operator defined in Lemma \ref{lemma C-1}, is invertible. This yields that $c(\om)=(I+S)^{-1}1$ satisfies the equation $(\ref{functional equation})$ and the function 
\[\varphi(\om,x)=\sum_{n=0}^\infty\frac{c(\sh^{-n}\om)\1_{[0, \tau_{\sh^{-n}\om}^{n}(1)]}(x)}{\be^{(n)}_{\sh^{-n}\om}},\]
for $(\om,x)\in\Om\times[0,1]$, satisfies $\cL_\om\vp_\om=\vp_{\sh\om}$
as in the proof of Theorem \ref{Thm C}.  
Let us define the operator $V:L^1(\mP)\to L^1(\mP)$ by
\[V(f)=\sum_{n=1}^\infty \frac{T_{\be_0}^n(1)}{\be_0^n}U_{\sh^{-1}}^n\]
for $f\in L^1(\mP)$. Since $||Vf||\leq \sum_{n=1}^\infty \frac{T_{\be_0}^n(1)}{\be_0^n}||f||$ for $f\in L^1(\mP)$, the operator $V$ is bounded. We can see that $V$ is linear by definition. 
The next lemma shows that the operator $I+V$ is invertible and the value of its operator norm
$||(I+V)^{-1}||$ is bounded by a certain positive constant related to the orbit of $1$ by the map $T_{\be_0}$.

\begin{lemma} \label{lemma D-1} 
(1)  The power series $\xi(z):=\xi_{\be_0}(z)=\sum_{n=0}^{\infty}T_{\be_0}^n(1)z^n/\be_0^n$ has $\be_0$ as its convergence radius and no zero in the unit closed ball $\{z\in\C\ ;\ |z|\leq1\}$.
In particular,  the Maclaurin series $\chi(z):=\chi_{\be_0}(z)=\sum_{n=0}^\infty t_n z^n$ of $1/\xi(z)$ has its convergence radius greater than $1$. 

(2) The power series $\chi(U_{\sh^{-1}})=\sum_{n=0}^\infty t_n U_{\sh^{-1}}^n$ is well-defined and we have 

\begin{equation}\label{inverse}
(I+V)\circ\chi(U_{\sh^{-1}})=\chi(U_{\sh^{-1}})\circ(I+V)=I.
\end{equation}
That is, $(I+V)^{-1}=\sum_{n=0}^\infty t_n U_{\sh^{-1}}^n$.
\end{lemma}

\begin{proof}
(1) Since $T_{\be_0}^n(1)\in[0,1]$ for $n\geq1$ we know that $\xi(z)$ has the convergence radius greater than or equal to $\be_0$. By using the relation 
\[\frac{[\be_0T_{\be_0}^{n-1}(1)]}{\be_0^n}=\frac{T_{\be_0}^{n-1}(1)}{\be^{n-1}_0}-\frac{T_{\be_0}^{n}(1)}{\be^{n}_0}\]
for $n\geq1$, which follows from the definition of the map $T_{\be_0}$, we have
\begin{equation}\label{Lem d1}
1-\sum_{n=1}^\infty\frac{[\be_0T_{\be_0}^{n-1}(1)]}{\be_0^n}z^n=(1-z)\xi(z)
\end{equation}
for $|z|<\be_0$. Note that 
\[\Biggl|\sum_{n=1}^\infty\frac{[\be_0 T_{\be_0}^{n-1}(1)]}{\be_0^n}z^n\Biggr|<\sum_{n=1}^\infty\frac{[\be_0 T_{\be_0}^{n-1}(1)]}{\be_0^n}=1\]
for $|z|\leq1$ except $z=1$ (see \cite{Fl-La}). Together with (\ref{Lem d1}), we have that $\xi(z)$ has no zero in $|z|\leq1$, which yields that the inverse of $\xi(z)$ can be defined as an analytic function in some open ball whose radius is greater than $1$. In particular, its Maclaurin series $\chi(z)=\sum_{n=0}^\infty t_n z^n$ has its convergence radius greater than $1$. 

(2) Note that $||U_{\sh^{-1}}||\leq1$. Then 
\[||\sum_{n=0}^\infty t_n U_{\sh^{-1}}^n||\leq\sum_{n=0}^\infty |t_n|<+\infty,\]
which ensures that the operator $\chi(U_{\sh^{-1}})$ is well-defined. 
Note that $I+V=\xi(U_{\sh^{-1}})$. Hence the equation (\ref{inverse}) follows from the fact that $\chi(z)$ is the inverse of $\xi(z)$ defined on the unit closed ball and 10 Theorem (b) in \cite{Du-Sc}. 

\end{proof}
Let $S:L^1(\mP)\to\ L^1(\mP)$ be the bounded linear operator defined in Lemma \ref{lemma C-1}. Using the fact that $V$ is invertible, we see that 
\[I+S=(I-(V-(I+S))V^{-1})V.\]
Since $||(V-(I+S))V^{-1}||\leq ||V-(I+S)||\cdot||V^{-1}||\leq||V-(I+S)||\chi(1)$, 
under the assumption that 
\[||V-(I+S)||<1/\chi(1),\]
the inverse of the operator $I+S$ is given by 
\[(I+S)^{-1}=V^{-1}\sum_{n=0}^\infty\Bigl\{(V-(I+S))V^{-1}\Bigr\}^n.\]
In the next lemma we give a sufficient condition to ensure that $||V-(I+S)||<1/\chi(1)$.

\begin{lemma}\label{lemma D-2}
Let $N\geq1$ be a positive integer satisfying $\sum_{n=N+1}^\infty2/\be_{0}^n<1/2\chi(1)$.
Then there is a positive number $\delta>0$ such that $[\be T_\be^i(1)]=[\be_0 T_{\be_0}^i(1)]$ for $1\leq i\leq N$ whenever $\be\in(\be_0-\delta,\be_0+\delta)$.
Set
\[\ep_1=\frac{1}{2}\min\{\langle\be_0\rangle, 1-\langle \be_0\rangle\},\ \ep_2=\frac{1}{4}(\be_0-1)([\be_0]-1+\langle\be_0\rangle/2)(\chi(1))^{-1}\]
and $\ep_0=\min\{\delta, \ep_1, \ep_2\}$, where $\langle y\rangle$ denotes the fractional part of $y\in\R$.
Then if $\be(\om)\in(\be_0-\ep_0, \be_0+\ep_0)$ for $\mP$-a.e. $\om\in\Om$, we have 
\[||V-(I+S)||<1/\chi(1).\]

\end{lemma}

\begin{proof}
For $n\geq1$, set
\[P_n (\be)=\be^n-\sum_{i=1}^{n}[\be_0 T^{i-1}_{\be_0}(1)]\be^{n-i}\]
for $\be\in([\be_0],[\be_0+1])$. Note that $P_n(\be_0)=T_{\be_0}^n(1)$ for $n\geq1$.
By the assumption that $\be_0$ is non-simple, we have that 
\[\frac{[\be_0 T^n_{\be_0}(1)]}{\be_0}<T_{\be_0}^n(1)=P_n(\be_0)<\frac{[\be_0 T^n_{\be_0}(1)]+1}{\be_0}\]
for $n\geq1$. By the continuity of the polynomial $P_1 (\be)$, we have that
there is a positive number $\del_1>0$ such that 
\[\frac{[\be_0T_{\be_0}(1)]}{\be}<T_\be(1)<\frac{[\be_0 T_{\be_0}(1)]+1}{\be}\]
whenever $\be\in(\be_0-\del_1, \be_0+\del_1)$. Inductively, we obtain that there is a positive number $\del$ such that 
\[\frac{[\be_0 T^n_{\be_0}(1)]}{\be}<T_{\be}^n(1)<\frac{[\be_0 T^n_{\be_0}(1)]+1}{\be}\]
for $1\leq n\leq N$ whenever $\be\in(\be_0-\del, \be_0+\delta)$, which gives the first statement.

We shall see that $||V-(I+S)||<1/\chi(1)$ if $\be(\om)\in(\be_0-\ep_0, \be_0+\ep_0)$ for $\mP$-a.e. $\om\in\Om$. We have 
\begin{linenomath}
\begin{align*}
||V-(I+S)||
&\leq\sum_{n=1}^\infty\Biggl|\frac{T_{\be_0}^n(1)}{\be_0^n}-\frac{\tau_{\sh^{-n}\om}^n(1)}{\be_{\sh^{-n}\om}^{(n)}}\Biggr| \\
&\leq\sum_{n=1}^\infty\Biggl|\frac{T_{\be_0}^n(1)}{\be_0^n}-\frac{\tau_{\sh^{-n}\om}^n(1)}{\be_0^n}\Biggr| 
+\sum_{n=1}^\infty\Biggl|\frac{\tau_{\sh^{-n}\om}^n(1)}{\be_0^n}-\frac{\tau_{\sh^{-n}\om}^n(1)}{\be_{\sh^{-n}\om}^{(n)}}\Biggr| \\
&\leq \sum_{n=1}^\infty\frac{1}{\be_0^n}|T_{\be_0}^n(1)-\tau_{\sh^{-n}\om}^n(1)|
+\sum_{n=1}^\infty\Biggl| \frac{1}{\be_0^n}-\frac{1}{\be_{\sh^{-n}\om}^{(n)}}\Biggr| \\
&\leq \sum_{n=1}^N\frac{1}{\be_0^n}|T_{\be_0}^n(1)-\tau_{\sh^{-n}\om}^n(1)|
+\sum_{n=N+1}^\infty\frac{2}{\be_0^n}+\sum_{n=1}^\infty\Biggl( \frac{1}{(\be_0-\ep_0)^n}-\frac{1}{\be_0^n}\Biggr). 
\end{align*}
\end{linenomath}
As $0<\ep_0\leq\delta$, we know that 
\begin{linenomath}
\begin{align*}
\sum_{n=1}^N\frac{1}{\be_0^n}|T_{\be_0}^n(1)-\tau_{\sh^{-n}\om}^n(1)|
&\leq \sum_{n=1}^N\frac{[\be_0]}{\be_0^n}\Biggl| \frac{1}{\be_0^n}-\frac{1}{\be_{\sh^{-n}\om}^{(n)}}\Biggr| \\
&\leq \sum_{n=1}^\infty\frac{[\be_0]}{\be_0^n}\Biggl( \frac{1}{(\be_0-\ep_0)^n}-\frac{1}{\be_0^n}\Biggr).
\end{align*}
\end{linenomath}
Then 
\begin{linenomath}
\begin{align*}
&\sum_{n=1}^N\frac{1}{\be_0^n}|T_{\be_0}^n(1)-\tau_{\sh^{-n}\om}^n(1)|
+\sum_{n=N+1}^\infty\frac{2}{\be_0^n}+\sum_{n=1}^\infty\Biggl( \frac{1}{(\be_0-\ep_0)^n}-\frac{1}{\be_0^n}\Biggr) \\
&\leq\frac{1}{2\chi(1)}+\sum_{n=1}^\infty\Biggl(1+\frac{[\be_0]}{\be_0^n}\Biggr)\Biggl( \frac{1}{(\be_0-\ep_0)^n}-\frac{1}{\be_0^n}\Biggr)  \\
&\leq\frac{1}{2\chi(1)}+2\Biggl(\frac{1}{\be_0-\ep_0-1}-\frac{1}{\be_0-1}\Biggr) \\
&\leq \frac{1}{2\chi(1)}+\frac{2\ep_0}{(\be_0-1)(\be_0-\ep_0-1)}.
\end{align*}
\end{linenomath}
Since $0<\ep_0\leq\ep_1$ and $0<\ep_0\leq\ep_2$, we obtain
\[\frac{2\ep_0}{(\be_0-1)(\be_0-\ep_0-1)}<\frac{2\ep_2}{(\be_0-1)([\be_0]-1+\langle\be_0\rangle/2)}<\frac{1}{2\chi(1)},\]
which gives the conclusion.

\end{proof}

\begin{theorem}\label{Thm D}
Let $\ep_0>0$ be the positive number defined as in Lemma \ref{lemma D-1}. If $\be(\om)\in(\be_0-\ep_0, \be_0+\ep_0)$ for $\mP$-a.e. $\om\in\Om$, then the bounded linear operator $I+S:L^1(\mP)\to L^1(\mP)$, where $S$ is the operator defined in Lemma \ref{lemma C-1}, is invertible. Then the function  
\[\varphi(\om,x)=\sum_{n=0}^\infty\frac{c(\sh^{-n}\om)\1_{[0, \tau_{\sh^{-n}\om}^{n}(1)]}(x)}{\be^{(n)}_{\sh^{-n}\om}}\]
for $(\om,x)\in\Om\times[0,1]$, where $c(\om)=(I+S)^{-1}1$, is a positive function in $L^1(\mP\times l)$ with $\int\vp d\mP\times l=1$ and satisfies 
\[\cL_{\om}\vp_\om=\vp_{\sh\om}.\]
That is, $\vp$ is the density function of $\mu$. 

\end{theorem}

\begin{proof}
By Lemma $\ref{lemma D-1}$ (2) and the equality 
\[I+S=(I-(V-(I+S))V^{-1})V,\]
if $||(V-(I+S))V^{-1}||<1$, the operator $I+S$ has the inverse. 
By Lemma $\ref{lemma D-2}$, we have $||(V-(I+S))V^{-1}||<1$
by the assumption $\be(\om)\in(\be_0-\ep_0, \be_0+\ep_0)$ for $\mP$-a.e. $\om\in\Om$.
Note that 
\[\int\vp\ d\mP\times l=\int(I+S)c \ d\mP\times l=1.\]
As in the proof of Theorem \ref{Thm C}, we can conclude that $\cL_{\om}\vp_\om=\vp_{\sh\om}$
if $c(\om)$ is in the kernel of $I-T$, which follows from Lemma \ref{lemma C-1}. This finishes the proof. 
\end{proof}

\end{subsection}
\end{section}

{\bf Acknowledgements.}
The author would like to thank Takehiko Morita, Hiroki Takahasi and Yuto Nakajima for their valuable comments. This work was supported by JSPS KAKENHI Grant Number 20K14331.

\bibliographystyle{amsplain}
\addcontentsline{toc}{section}{References}

\end{document}